\documentclass[11pt]{article}
\usepackage[utf8]{inputenc}
\usepackage[english]{babel}
\usepackage[T1]{fontenc}
\usepackage{natbib}
\usepackage{amssymb}
\usepackage{amsmath}
\usepackage{amsthm}
\usepackage{dsfont}
\newtheorem{thm}{Theorem}
\newtheorem{prop}{Proposition}
\newtheorem{remm}{Remark}
\newtheorem{Lem}{Lemma}
\newtheorem{Cor}{Corollary}
\usepackage{amssymb,amsmath}
\usepackage{float}
\usepackage{graphicx}
\usepackage{enumerate}
\newcommand{\cov}{{\rm Cov}}
\newcommand{\var}{{\rm Var}}
\def\cum{\mathrm{cum}}
\usepackage{comment}
\def\E{\mathbb{E}}

\def\N{\mathbb{N}}
\def\R{\mathbb{R}}

\title{Inference for continuous-time long memory randomly sampled processes}
\author{ Mohamedou Ould Haye$^1$, Anne Philippe$^2$\footnote{corresponding author : {anne.philippe@univ-nantes.fr}} and
 Caroline Robet$^2$
}
\date{$^1$\small School of Mathematics and Statistics. \\
Carleton University, 1125 Colonel By Dr. Ottawa, ON, Canada, K1S 5B6
\\ $^2$ Laboratoire de Math\'{e}matiques Jean Leray,\\
2 rue de la Houssiniere, Universit\'{e} de Nantes, 44 322 Nantes France.}
\begin{document}
\maketitle
\begin{abstract} 
From a continuous-time long memory stochastic process, a discrete-time randomly sampled one is drawn using a renewal sampling process. We establish the existence of the spectral density of the sampled process, and we give its expression in terms of that of the initial process. 
We also investigate different aspects of the statistical inference on the sampled process. In particular, we obtain asymptotic results for the periodogram, the local Whittle estimator of the memory parameter and the long run variance of partial sums. We mainly focus on Gaussian continuous-time process. The challenge being that the randomly sampled process will no longer be jointly Gaussian.
\end{abstract}

\section{Introduction}
Irregularly observed time series occur in many fields such as astronomy, finance, environmental, and biomedical sciences. 
Discretization of a continuous time process can produce unevenly time series. For example, physiological signals such as electromyography (EMG), electrocardiogram (ECG), as well as heartbeats (see e.g. \cite{bardet}) are measured at non regularly spaced times. In finance, market prices are tick-by-tick data; a tick being happening randomly, depending for instance on transaction prices. Such data constitute an other example of irregularly spaced time series (see e.g. \cite{Dacorogna}).

Irregular sampling interval is also used whenever one has some uncertainty surounding actual dates such as paleoclimatic time series (see e.g. \cite{Thom}), temperature and CO2 measurements data studied by \cite{Nieto}. In these instances, we do not control the way data are observed, as they are recorded at irregular time points. A common approach consists in fitting a continuous time process to discrete data (see for instance \cite{jones19855}).

Statistical tools available to handle unevenly time series are essentially developed for short range dependence (see e.g. \cite{Li} and references therein). We can also refer to numerous papers in astronomy, that focus on spectrum estimation (see e.g. \cite{Thiebaut}).

To the best of our knowledge, few results are available when the continuous-time embedding process has a long memory. Actually, long memory statistical inference for continuous-time models is generally
built upon a deterministically sampled process (see \cite{tsai1,tsai2,chambers,comte1}).
However, as in the examples previously cited, in several applied contexts one has to deal with random sampling from a continuous process. \cite{Phi2} studied randomly-spaced observations, using a renewal process as a sampling tool. They showed that the intensity of the long memory is preserved when the distribution of sampling intervals has a finite moment, but there are also situations where a reduction of the long memory is observed. Consequently, the continuous time memory parameter cannot be estimated without a prior information on the sampling process. \cite{bardet} studied spectral density estimation of continuous-time Gaussian processes with stationary increments observed at random times. \\
In this paper, we describe the spectral properties of the resulting discrete-time-indexed randomly sampled process and we provide more explicit expressions for the spectral density of the sampled process. 
We mention that \cite{Phi1} addressed resampling from a discrete-time process and obtained the existence of the spectral density. However, their spectral density expression is less explicit since it is expressed as a non explicit limit of an integral.

Most of existing long memory inferential techniques assume that
the process is a subordinated Gaussian/linear one. \cite{Phi2} established a rather surprising characteristic consisting in the loss of the joint-Gaussianity of the sampled process when the original process was Gaussian. Therefore we cannot apply such results to our  sampled processes that are neither Gaussian nor strongly linear. We study some aspects of the inference via spectral approaches. In particular, to establish the consistency of long memory parameter's local Whittle estimator using \cite{dalla}'s assumptions for nonlinear long memory processes.\\\\
We now describe our sampling model. We start with $X=(X_t)_{t\in \R ^+}$, a continuous time process and a renewal process $(T_n)_{n\geq 0}$. 
We study the discrete-time indexed process $Y=(Y_n)_{n\ge 1}$ defined by
\begin{equation}\label{def}
Y_n=X_{T_n}\quad n=1,2,\ldots.
\end{equation}
We want to emphasise that the sampling process $T_n$ is not observed. Throughout this paper, we will assume that, and refer to
\begin{itemize}
 \item[$H_X$ : ] $X$ is second-order stationary continuous time process with auto-covariance function $\sigma_X$ and having a spectral density $f_X$: for all $t\in \R$
 \begin{equation}\label{fX} \sigma_X(t) = \int_{-\infty}^\infty e^{i\lambda t}f_X(\lambda)d\lambda.\end{equation} 
 \item[$H_T$ : ] $(T_n)_{n\geq 0}$ independent of $X$ and of i.i.d. increments $T_{j+1}-T_j=\Delta_j\geq 0$ non degenerate with cumulative distribution function $S$ and we let $T_0=0$.
\end{itemize}
 We impose this specific 
 initialization $T_0=0$ only to simplify our notations since it implies
 that $\Delta_j = T_{j+1}-T_j$ for all $j\in\N$. However, all the results remain true if we take $T_0 = \Delta_0 $ and
 $\Delta_j = T_{j}-T_{j-1}$, for $j\geq 1$.

 The rest of the paper is organized as follows. Section \ref{spectral_density} presents results on the existence of a spectral density for the process $Y$ when the spectrum of $X$ is absolutely continuous. We also provide an integral representation of such density. In Section \ref{sec:perio}, we establish the asymptotic distribution of the normalized periodogram of the sampled process. In Section \ref{sec:inf}, we show the consistency of $Y$-based local Whittle memory estimator. We also study the estimation of the so-called long-run variance.

 \section{Spectral density function of sampled process}\label{spectral_density}
 
Under the assumptions $ H_X$ and $H_T$, \cite{Phi2} show that if $X$ is stationary then so is $Y$. Moreover, its covariance function is of the following form 
\begin{eqnarray}\label{gammak2}
\sigma_Y(j)=\textrm{Cov}(Y_1,Y_{j+1})=\mathbb{E}(\sigma_X({T_j})).
\end{eqnarray}
Note that the independence of $X$ and the renewal process imposed in $H_T$, is required to get \eqref{gammak2}.
In the next proposition, we prove that the existence of the spectral density is preserved
by random sampling and we establish the link
between the spectral densities of processes $X$ and $Y$.
\begin{prop}

\label{sigma_Y} Assume that the continuous-time process $X$ satisfies $H_X$ and that $H_T$ holds. Then, the discrete-time process $Y$ admits a spectral density and it is given by the following formula
\begin{equation}\label{limitesigma}
f_Y(x)=\frac{1}{2\pi}\int_{-\infty}^\infty p(x,\Psi_S(\lambda))f_X(\lambda)d\lambda,
\end{equation}
where $\Psi_S$ is the characteristic function of the cumulative distribution function $S$ defined in $H_T$ and
$$
p(x,z)=\frac{1-\vert z\vert^2}{\vert 1-e^{-ix}z\vert ^2},\qquad \vert z\vert<1
$$
is the well known Poisson kernel.
\end{prop}
\begin{proof} According to the stationarity property and \eqref{gammak2} proved in \cite{Phi2} and the existence of the spectral density $f_X$ in \eqref{fX}, the covariance function of $Y$ can be computed via Fubini's theorem as follows:
\begin{eqnarray}\label{gammak}\lefteqn{
\sigma_Y(j)=\mathbb{E}(\sigma_X({T_j}))=\mathbb{E}\left(\int_{-\infty}^\infty e^{i\lambda T_j}f_X(\lambda)d\lambda\right)}\nonumber\\
&&=\int_{-\infty}^\infty\left(\mathbb{E}\left(e^{i\lambda T_j}\right)\right)f_X(\lambda)d\lambda=\int_{-\infty}^\infty\left(\Psi_S(\lambda)\right)^jf_X(\lambda)d\lambda.
\end{eqnarray}
To prove (\ref{limitesigma}), it will suffice to show that for every $j\ge0$,
\begin{equation}\label{cov00}
\sigma_Y(j)=\int_{-\pi}^\pi e^{ijx}f_Y(x)dx,
\end{equation}
as $f_Y$ defined by (\ref{limitesigma}) is clearly an even function.
For this, we will use the following Poisson integral formula for the disk: if $u$ is an analytic function on the disk $\vert z\vert<1$ and continuous on $\vert z\vert=1$ then its real and imaginary parts are harmonic and therefore for $\vert z\vert<1$, we have
$$
u(z)=\frac{1}{2\pi}\int_{-\pi}^\pi u(e^{ix})p(x,z)dx.
$$
Applying the above with $u(z)=z^j$, where $j$ is a fixed nonnegative integer, we get
\begin{equation}\label{zj}
 z^j=\frac{1}{2\pi}\int_{-\pi}^\pi e^{ijx}p(x,z)dx,\quad\textrm{for all }\vert z\vert<1,
\end{equation}
 and since for Lebesgue a.e. $\lambda$, $\vert\Psi_S(\lambda)\vert<1$ ($S$ being non degenerated), then for a.e. $\lambda$,
\begin{equation}\label{jj}
(\Psi_S(\lambda))^j=\frac{1}{2\pi}\int_{-\pi}^\pi e^{ijx}p(x,\Psi_S(\lambda))dx.
\end{equation}
Also taking $j=0$ in \eqref{zj}, we get 
$$
\frac{1}{2\pi}\int_{-\pi}^\pi p(x,z)dx=1\quad\textrm{for all }\vert z\vert<1. 
$$
Hence, by Fubini's theorem, we see that $f_Y$, as given in \eqref{limitesigma}, is integrable on $[-\pi, \pi]$. Applying Fubini's theorem once again and 
substituting (\ref{jj}) in (\ref{gammak}), we immediately get (\ref{cov00}).
\end{proof}

The following corollary gives a precise expression of the spectral density of $Y$ in the most common case of Poisson renewal process.
\begin{Cor}
 Assume that the continuous-time process $X$ satisfies $H_X$ and that $(T_n)$ is a Poisson renewal process with rate 1, independent of $X$. If $\lambda^2f_X(\lambda)$ is bounded and continuous on the real line then 
\begin{equation}\label{poi1}
f_Y(x)=\frac{u(\sin x,1-\cos x)}{2(1-\cos x)},\qquad\textrm{for a.e. }x\textrm{ in the interval }(-\pi,\pi),
\end{equation}
where $u(x,y)$ is the harmonic function on the upper half plane with boundary condition $u(x,0)=x^2f_X(x)$. In particular, both spectral densities are equivalent near zero, i.e., $f_Y(x)\sim f_X(x)$ as $x\to0$.
\end{Cor}
\begin{proof}
The exponential distribution has characteristic function $(1-i\lambda)^{-1}$ and hence from Proposition \ref{sigma_Y}, we can easily derive that 
\begin{eqnarray}\lefteqn{
f_Y(x)=\frac{1}{2\pi}\int_{-\infty}^\infty\left(\frac{\lambda^2}{(\lambda-\sin x)^2+(1-\cos x)^2}\right)f_X(\lambda)d\lambda} \label{fyPois1}\\
&&=
\frac{1}{2(1-\cos x)}\frac{1}{\pi}\int_{-\infty}^\infty\left(\frac{1-\cos x}{(\lambda-\sin x)^2+(1-\cos x)^2}\right)\lambda^2f_X(\lambda)d\lambda. \label{fyPois}
\end{eqnarray}
In the above we recognise the well known Poisson integral formula for the upper half plane for the function: $x\mapsto x^2f_X(x)$:
if $g$ is continuous and bounded on the real line then the function defined by
\begin{equation}\label{u}
u(x,y):=\frac{1}{\pi}\int_{-\infty}^\infty\left(\frac{y}{(x-\lambda)^2+y^2}\right)g(\lambda)d\lambda
\end{equation}
is harmonic on the upper half plane and satisfies $u(x,0)=g(x)$ (see for example the result 7.3 on page 147 of \cite{Axler}) and $\frac{u(x,y)}{g(x)} \to 1$ uniformly in $x$ as $y\to0$. Combining \eqref{fyPois} and \eqref{u} we get the stated result. 
\end{proof}\noindent
The next proposition precises the behaviour of the spectral density of sampled process $Y$ near zero, given in the previous corollary, under mild semi parametric conditions on the spectral density of the original process $X$.
 \begin{prop}\label{spectral-lemma}
 Assume that $T_n$ is a Poisson process independent of $X$ with rate 1 and that $X$ satisfies $H_X$ with spectral density of the form
 \begin{equation}\label{eq:lmspe}
 f_X(\lambda)=\vert\lambda\vert^{-2d}\phi(\lambda),
 \end{equation} with $0<d<1/2$,
 $\phi(0)\neq0$ and $\phi$ is continuous on $[-1,1]$ and differentiable on $(-1 , 1 )$. Then
 \begin{equation}\label{spectral-form}
 f_Y(x)=|x|^{-2d}f_Y^*(x)
 \end{equation}
 with $f_Y^*$ is positive continuous on $[-\pi,\pi]$ and
 \begin{equation}\label{dala}
 f_Y^*(x)=\phi(0)+\frac{\sigma_X(0)}{2\pi}|x|^{2d}+o(|x|^{2d}),\qquad\textrm{as }x\to0.
 \end{equation}
 \end{prop}
\begin{proof}
Since $f_Y$ is even, we will consider $x\in(0,\pi]$. From \eqref{fyPois1}, we have
\begin{eqnarray}\label{no-zero}
f_Y(x)=\frac{1}{2\pi}\int_0^\infty\frac{\lambda^2}{(\lambda-\sin x)^2+(1-\cos x)^2}f_X(\lambda)d\lambda+
\frac{1}{2\pi}\int_0^\infty\frac{\lambda^2}{(\lambda+\sin x)^2+(1-\cos x)^2}f_X(\lambda)d\lambda.
\end{eqnarray}
 We study both integrals in (\ref{no-zero}) near $x=0$.
 \begin{equation}\label{first1}
 \int_0^\infty\frac{\lambda^2}{(\lambda+\sin x)^2+(1-\cos x)^2}f_X(\lambda)d\lambda\overset{x\to0}{\to}\int_0^\infty f_X(\lambda)d\lambda=\frac{\sigma_X(0)}{2}
 \end{equation}
 since for fixed $\lambda$, as $x\to0$, the integrand (in the left-hand side) clearly increases towards $f_X(\lambda)$. Let us deal with the first integral in (\ref{no-zero}).
 \begin{eqnarray}\lefteqn{
\int_0^\infty\frac{\lambda^2}{(\lambda-\sin x)^2+(1-\cos x)^2}f_X(\lambda)d\lambda=\int_0^{\sin x}\frac{\lambda^2}{(\lambda-\sin x)^2+(1-\cos x)^2}f_X(\lambda)d\lambda+} \nonumber\\
&&\int_{\sin x}^{2\sin x}\frac{\lambda^2}{(\lambda-\sin x)^2+(1-\cos x)^2}f_X(\lambda)d\lambda+\int_{2\sin x}^\infty\frac{\lambda^2}{(\lambda-\sin x)^2+(1-\cos x)^2}f_X(\lambda)d\lambda. \label{decommp3}
\end{eqnarray}
 Using the fact that $f_X(\lambda)=\lambda^{-2d}\phi(\lambda)$ and $\sin^2(x/2)=(1-\cos x)/2$ and putting $\lambda=t\sin x$, we obtain for the first integral in the right hand side above, with some $u(t)\in(0,1)$ and $v(t)\in(0,1)$,
\begin{eqnarray}\lefteqn{
\int_0^{\sin x}\frac{\lambda^2}{(\lambda-\sin x)^2+4(\sin (x/2))^4}f_X(\lambda)d\lambda=(\sin x)^{-2d}\int_0^1\frac{t^{2-2d}\sin x}{(1-t)^2+\tan^2(x/2)}\phi(t\sin x)dt}\nonumber\\
&&=(\sin x)^{-2d}\int_0^1\frac{(1-t)^{2-2d}\sin
 x}{t^2+\tan^2(x/2)}\phi((1-t)\sin x)dt \nonumber\\
&&=(\sin x)^{-2d}\int_0^1(1-(2-2d)(1-u(t))^{1-2d}t)\frac{\sin x}{t^2+\tan^2(x/2)}(\phi(0)+\phi'(v(t))\sin x(1-t))dt\nonumber\\
&&=(\sin x)^{-2d}\left[\int_0^1\frac{\phi(0)\sin
 x}{t^2+\tan^2(x/2)}dt+O\left(\sin
 x\int_0^1\frac{t}{t^2+\tan^2(x/2)}dt+\sin x\int_0^1\frac{\sin
 x}{t^2+\tan^2(x/2)}dt\right)\right]. \nonumber
 \end{eqnarray}
 Putting $ t=u\tan(x/2)$ the right-hand-side of the last equation
 is equal to
\begin{eqnarray*}
&&(\sin
 x)^{-2d}\Bigg[2\phi(0)\cos^2(x/2)\int_0^{1/\tan(x/2)}\frac{1}{u^2+1}du
 +
 \\
&& \hskip 4cm O\left(\sin x\int_0^1\frac{t}{t^2+\tan^2(x/2)}dt+\sin x\int_0^1\frac{\sin x}{t^2+\tan^2(x/2)}dt\right)\Bigg]\nonumber\\
&&=(\sin
 x)^{-2d}\left[2\phi(0)\cos^2(x/2)\arctan(1/\tan(x/2))+O\left(x(\log(1+x^2)-\log
 x)+2x\arctan(2/x)\right)\right]\nonumber
\end{eqnarray*}
Then
\begin{equation}\label{first1-2}
\int_0^{\sin x}\frac{\lambda^2}{(\lambda-\sin x)^2+4(\sin
 (x/2))^4}f_X(\lambda)d\lambda
=x^{-2d}\left(\phi(0)\pi+O(x\log x)\right),\qquad\textrm{as
}x\to0.
\end{equation}
Similarly, we have
\begin{eqnarray}\label{first2-2}\lefteqn{
\int_{\sin x}^{2\sin x}\frac{\lambda^2}{(\lambda-\sin x)^2+(\sin x/2))^4}f_X(\lambda)d\lambda=(\sin x)^{-2d}\int_1^2\frac{t^{2-2d}\sin x}{(1-t)^2+\tan^2(x/2)}\phi(t\sin x)dt}\nonumber\\
&&=(\sin x)^{-2d}\int_0^1\frac{(1+t)^{2-2d}\sin x}{t^2+\tan^2(x/2)}\phi((1+t)\sin x)dt\nonumber\\
&&=(\sin x)^{-2d}\int_0^1(1+(2-2d)(1+u(t))^{1-2d}t)\frac{\sin x}{t^2+\tan^2(x/2)}(\phi(0)+\phi'(v(t))\sin x(1+t))dt\qquad\nonumber\\
&&= x^{-2d}\left(\phi(0)\pi+O(x\log x)\right).
\end{eqnarray}
Then, we have as $x\to0$
\begin{equation}\label{first2}
\int_{2\sin x}^\infty\frac{\lambda^2}{(\lambda-\sin x)^2+\tan^2(x/2)}f_X(\lambda)d\lambda\to\int_0^\infty f_X(\lambda)d\lambda=\frac{\sigma_X(0)}{2},
\end{equation}
since the integrand is bounded uniformly in $x$ by $4f_X(\lambda)$ and converges (as $x\to0$) to $f_X(\lambda)$ and hence we can apply Lebesgue's theorem. Combining (\ref{first1}) and (\ref{first2}) as well as (\ref{first1-2}) and (\ref{first2-2}), we obtain that
$$
f_Y(x)=x^{-2d}f^*_Y(x),\qquad f^*_Y(x)=\phi(0)+\frac{\sigma_X(0)}{2\pi}x^{2d}+o\left(x^{2d}\right)\quad\textrm{as }x\to0. 
$$
Moreover, $f_Y^*$ is continuous and positive on $[-\pi,,\pi]$. Indeed, the continuity of $f_Y^*$ follows from the fact that the 2nd integrand in the right hand side of \eqref{no-zero} is continuous and uniformly bounded in $x$ by $4f_X(\lambda)$ which is integrable. As for the first integral in the right hand side of \eqref{no-zero}, after splitting it into three terms as in \eqref{decommp3} and multiplying it by $x^{2d}$, we see that Lebesgue's dominated convergence theorem still applies. This completes the proof of Proposition \ref{spectral-lemma}. 
 \end{proof}
We now present a lemma that gives a quite precise expression of the
covariance function of $X$ from its spectral density. We will be
imposing the following condition on $f_X$.
\vskip.3cm 
\noindent \textbf{Condition $H_f$:}
$f_X(\lambda)=c\vert\lambda\vert^{-2d}(1-h(\lambda))$, $0<d<1/2$, where $h$ is
a nondecreasing function with $h(0)=0$ and $h(x)\to1$ as $x\to\infty$ and $h$ is differentiable at 0.
We notice that condition $H_f$ is not one of the usual slowly varying type conditions for Tauberian and Abelian theorems in the context of long range dependence (see \cite{leonenko}). However, it guarantees a uniform control of the remainder $g(x)$ in (\ref{cov-behav}) rather than at infinity only.
\begin{remm}\label{rem-change}
 If the spectral density $f_X$ satisfies $H_f$ instead of (\ref{eq:lmspe}), then 
 Proposition \ref{spectral-lemma} still holds with
 $c:=c(d)$ instead of $\phi(0)$. The proof is essentially the same and is
 omitted.
\end{remm}
\begin{Lem}\label{X-cov-density}
Assume that condition $H_f$ is satisfied. Then, there exist positive constants $C(d)$ and $ c(d)$ such that for all $x>0$,
\begin{equation}\label{cov-behav}
\sigma_X(x)=c(d)x^{2d-1}+g(x),
\end{equation}
with $\vert g(x)\vert\le \frac{C(d)}{\vert x\vert}.$
\end{Lem}
\begin{proof} Let $x>0$ be fixed.
Since  $f_X$ is even we have,
$$
\sigma_X(x)=2\int_0^\infty \cos(x\lambda)f_X(\lambda)d\lambda.
$$
Without loss of generality, we take $2c=1$ in $H_f$
and by the formula 3.761.9 of \cite{gra}
$$
\int_0^\infty\cos(x\lambda)\lambda^{-2d}d\lambda=\Gamma(1-2d)\sin(\pi d) x^{2d-1}=:c(d)x^{2d-1}.
$$
Therefore, it remains to show that for some $C(d)>0$,
$$
\left\vert\int_0^\infty \cos(\lambda x)\lambda^{-2d}h(\lambda)d\lambda\right\vert\le C(d)x^{-1}.
$$
The rest of the proof relies on applying integration by parts for Stieltjes integrals. \\ 
Let $dU(\lambda)=\cos(\lambda x)\lambda^{-2d}$. We have (by one integration by parts)
\begin{eqnarray*}\lefteqn{
U(t)=\int_0^t\cos(\lambda x)\lambda^{-2d}d\lambda}\\
&&=\frac{1}{x}
\left[\lambda^{-2d}\sin(\lambda x)\right]_{\lambda=0}^{\lambda=t}+\frac{2d}{x}\int_0^t\lambda^{-2d-1}\sin(\lambda x)d\lambda\\
&&=\frac{1}{x}\left(t^{-2d}\sin(tx)+2d\int_0^t\lambda^{-2d-1}\sin(\lambda x)d\lambda\right),
\end{eqnarray*}
clearly $U$ is bounded and
$$
\underset{t\to\infty}{\lim}U(t)=\frac{2d}{x}\int_0^\infty\lambda^{-2d-1}\sin(\lambda x)d\lambda.
$$
Using the fact that $h$ is nondecreasing, $h(\lambda)\to1$, as $\lambda\to\infty$, and $h(0)=0$, we obtain (via integration by parts at some steps in the calculation below 
\begin{equation}
 \int_a^b \cos(\lambda x)\lambda^{-2d}h(\lambda)d\lambda=\int_a^ b h(\lambda)dU(\lambda)=
\left[U(\lambda)h(\lambda)\right]_a^b-\int_a^b U(\lambda)dh(\lambda), 
\end{equation}
with 
$$\left[U(\lambda)h(\lambda)\right]_a^b \xrightarrow[b\to\infty]{a\to 0 }\frac{2d}{x}\int_0^\infty\lambda^{-2d-1}\sin(\lambda x)d\lambda, $$ 

\begin{eqnarray*} 
-\int_a^b U(\lambda)dh(\lambda) &&= - \frac{1}{x}\int_a^b\left(\lambda^{-2d}\sin(\lambda x)+2d\int_0^\lambda u^{-2d-1}\sin(ux)du\right)dh(\lambda)\\
&&=- \frac{1}{x}\int_a^b\lambda^{-2d}\sin(\lambda x)dh(\lambda)-\frac{2d}{x}\int_a^b\left(\int_0^\lambda u^{-2d-1}\sin(ux)du\right)dh(\lambda),\\
\end{eqnarray*}
also 
\begin{eqnarray*} 
&& -\frac{2d}{x}\int_a^b\left(\int_0^\lambda u^{-2d-1}\sin(ux)du\right)dh(\lambda) \\ && = -\frac{2d}{x} h(b) \int_0^b \lambda^{-2d-1}\sin(\lambda x)d\lambda +\frac{2d}{x} h(a) \int_0^a\lambda^{-2d-1}\sin(\lambda x)d\lambda 
+\frac{2d}{x}\int_a^b \lambda^{-2d-1}\sin(\lambda x)h(\lambda)d\lambda.\\
\end{eqnarray*}
Since 
$$-\frac{2d}{x} h(b) \int_0^b \lambda^{-2d-1}\sin(\lambda x)d\lambda +\frac{2d}{x} h(a) \int_0^a\lambda^{-2d-1}\sin(\lambda x)d\lambda \xrightarrow[b\to\infty]{a\to 0 } - \frac{2d}{x}\int_0^\infty\lambda^{-2d-1}\sin(\lambda x)d\lambda, $$ 
 $ \int_a^b U(\lambda)dh(\lambda)$ has the same limit as 
$$- \frac{1}{x}\int_a^b\lambda^{-2d}\sin(\lambda x)dh(\lambda) + \frac{2d}{x}\int_a^b \lambda^{-2d-1}\sin(\lambda x)h(\lambda)d\lambda,$$
as $a\to 0 $ and $b\to\infty$. 

\begin{eqnarray*} && \left| - \frac{1}{x}\int_a^b\lambda^{-2d}\sin(\lambda x)dh(\lambda) + \frac{2d}{x}\int_a^b \lambda^{-2d-1}\sin(\lambda x)h(\lambda)d\lambda\right| \\
 &&\leq 
\frac{1}{x}\int_a^b \lambda^{-2d}dh(\lambda)+\frac{2d}{x}\int_a^b\lambda^{-2d-1}h(\lambda)d\lambda\\
&& = \frac{1}{x} \left(h(b) b^{-2d} -h(a) a^{-2d}\right) + \frac{2d}{x}\int_a^b \lambda^{-2d-1}h(\lambda)d\lambda+\frac{2d}{x}\int_a^b \lambda^{-2d-1}h(\lambda)d\lambda\\
&& =\frac{1}{x} \left(h(b) b^{-2d} -h(a) a^{-2d}\right) + \frac{4d}{x}\int_a^b\lambda^{-2d-1}h(\lambda)d\lambda \xrightarrow[b\to 0] {a\to 0} \frac{4d}{x}\int_0^\infty\lambda^{-2d-1}h(\lambda)d\lambda :=\frac{C(d)}{x}.
\end{eqnarray*}
We note that the integral above is indeed finite since $h$ is a bounded function, $h(0)=0$, and is differentiable at zero.

The proof of Lemma \ref{X-cov-density} is now complete.
\end{proof}

\begin{Cor}\label{var-bound}
If $T_n$ is a Poisson process and $f_X$ satisfies condition $H_f$ then 
$$
\textrm{Var}\left(\sigma_X(T_r)\right)=O\left(r^{-\alpha}\right),\qquad\textrm{as }r\to\infty,
$$
where $\alpha=\min(2,3-4d)$.
\end{Cor}
\begin{proof} We have from the previous lemma,
\begin{eqnarray*}\lefteqn{
\textrm{Var}\left(\sigma_X(T_r)\right)=\textrm{Var}\left(c(d)T_r^{2d-1}+g(T_r)\right)}\\
&&=c^2(d)\textrm{Var}\left(T_r^{2d-1}\right)+\textrm{Var}\left(g(T_r)\right)+2c(d)\textrm{Cov}\left(T_r^{2d-1},g(T_r)\right)\\
&&\le c^2(d)\textrm{Var}\left(T_r^{2d-1}\right)+C^2(d)\mathbb{E}\left(T_r^{-2}\right)+2c(d)C(d)\left[\textrm{Var}\left(T_r^{2d-1}\right)\right]^{1/2}
\left[\mathbb{E}\left(T_r^{-2}\right)\right]^{1/2}.
\end{eqnarray*}
For $r\ge3$, as $T_r$ has Gamma distribution with parameters $(r,1)$, we have
$$
\mathbb{E}\left(T_r^{-2}\right)=\int_0^\infty \frac{x^{r-2-1}}{\Gamma(r)}e^{-x}dx=\frac{\Gamma(r-2)}{\Gamma(r)}=\frac{1}{(r-1)(r-2)}=O(r^{-2}).
$$
Also,
$$
\textrm{Var}\left(T_r^{2d-1}\right)=\mathbb{E}\left(T_r^{4d-2}\right)-\left(\mathbb{E}\left(T_r^{2d-1}\right)\right)^2=
\frac{\Gamma(r-2+4d)}{\Gamma(r)}-\left(\frac{\Gamma(r-1+2d)}{\Gamma(r)}\right)^2
$$
We know that as $r\to\infty$,
$$
\frac{\Gamma(r-a)}{\Gamma(r)}=r^{-a}\left(1-\frac{a(-a+1)}{2r}+O\left(\frac{1}{r^2}\right)\right),
$$
and therefore we obtain that
$$
\textrm{Var}\left(T_r^{2d-1}\right)=(1-2d)^2r^{-2(1-2d)-1}+o\left(r^{-2(1-2d)-1}\right)=O\left(n^{-\alpha}\right),
$$
which completes the proof of the corollary.
\end{proof}

\section{Asymptotic theory of the periodogram}\label{sec:perio}
We consider in this section a stationary long memory zero-mean Gaussian process $X=(X_t)_{ t\in\mathbb{R}^+}$ having a spectral density of the form (\ref{eq:lmspe}). Let $Y = (X_{T_n})_{n\in \mathbb{N}}$, where $(T_n)_{n\in \mathbb{N}}$ is a Poisson process with rate equal 1 (actually any rate will do). As shown in \cite{Phi2} and in contrast with the original process $X$, while $Y$ remains marginally normally distributed, it is no longer jointly Gaussian and, as a result, $Y$ is not a linear process.

In this section, we extend some well-known facts about periodogram properties to the randomly sampled processes $Y$. In particular, our main result will be to establish that the normalized periodogram of $Y$ will asymptotically converge to a weighted $\chi^2$ distribution. 
\begin{thm}\label{theo1}
Assume that $X$ is a stationary Gaussian process satisfying $H_f$ and let $Y = (X_{T_n})_{n\in \mathbb{N}}$ where $(T_n)_{n\in \mathbb{N}}$ is a Poisson process with rate equal 1. Let

$$I_n(\lambda_j)=\frac{1}{2\pi n}\left\vert\sum_{k=1}^nY_ke^{ik\lambda_j}\right\vert^2,$$
be the periodogram of $Y_1,\ldots,Y_n$ at Fourier frequency $\lambda_j=2\pi j/n$ for $j\in \lbrace 1,~\dots~, \lfloor n/2\rfloor \rbrace$. Then, we have for any fixed number of Fourier frequencies $\nu \geq 1$, and any $j_1,\dots ,j_\nu \in \lbrace 1, \dots ,\lfloor n/2\rfloor \rbrace$ all distinct integers
\begin{equation}\label{limit}
\left(\frac{I_n(\lambda_{j_1})}{f_Y(\lambda_{j_1})},\cdots,\frac{I_n(\lambda_{j_\nu})}{f_Y(\lambda_{j_\nu})}\right)\overset{\mathcal{D}}{\to}
\left(L_{j_1}(d)[Z^2_1(j_1)+Z^2_2(j_1)],\cdots,L_{j_\nu }(d)[Z^2_1(j_\nu)+Z^2_2(j_\nu )]\right),
\end{equation}
where $(Z_1(1),Z_2(1),\ldots,Z_1([ n/2] ),Z_2([ n/2] )$ is a zero-mean Gaussian vector, with
$Z_1(j),Z_2(k)$ are independent for all $j,k=1,\ldots, [n/2] $ and
\begin{equation}\label{eq1}
\textrm{Var}(Z_1(j))=\frac{1}{2}-\frac{R_j(d)}{L_j(d)}
\end{equation}
and
\begin{equation}\label{eq2}
\textrm{Var}(Z_2(j))=\frac{1}{2}+\frac{R_j(d)}{L_j(d)},
\end{equation}
and for $j\neq k$,
\begin{equation}\label{eq3}
\textrm{Cov}(Z_1(j),Z_1(k))=\frac{L_{j,k}(d)-R_{j,k}(d)}{\sqrt{L_j(d)L_{k}(d)}}
\end{equation}
\begin{equation}\label{eq4}
\textrm{Cov}(Z_2(j),Z_2(k))=\frac{L_{j,k}(d)+R_{j,k}(d)}{\sqrt{L_j(d)L_{k}(d)}},
\end{equation}
with
\begin{equation}\label{eq5}
L_j(d)=\frac{2}{\pi}\int_{-\infty}^\infty\frac{\sin^2(\lambda/2)}{(2\pi j-\lambda)^2}\left\vert\frac{\lambda}{2\pi j}\right\vert^{-2d}d\lambda,
\end{equation}
\begin{equation}\label{eq6}
R_j(d)=\frac{1}{\pi}\int_{-\infty}^\infty\frac{\sin^2(\lambda/2)}{(2\pi j-\lambda)(2\pi j+\lambda)}\left\vert\frac{\lambda}{2\pi j}\right\vert^{-2d}d\lambda,
\end{equation}
\begin{equation}\label{eq7}
L_{j,k}(d)=\frac{(jk)^d}{\pi}\int_{-\infty}^\infty\frac{\sin^2(\lambda/2)}{(2\pi k-\lambda)(2\pi j-\lambda)}\left\vert\frac{\lambda}{2\pi}\right\vert^{-2d}d\lambda,
\end{equation}
and
\begin{equation}\label{eq8}
R_{j,k}(d)=\frac{(jk)^d}{\pi}\int_{-\infty}^\infty\frac{\sin^2(\lambda/2)}{(2\pi k+\lambda)(2\pi j-\lambda)}\left\vert\frac{\lambda}{2\pi}\right\vert^{-2d}d\lambda.
\end{equation}
\end{thm}
\begin{proof}
We will prove the broader result

\begin{align}\label{clt}
Z_n: = \Bigg(\frac{1}{\sqrt{2\pi nf_Y(\lambda_{j_1})}}&\sum_{r=1}^n\cos(r\lambda_{j_1})X_{T_r},\frac{1}{\sqrt{2\pi nf_Y(\lambda_{j_1})}}\sum_{r=1}^n \sin(r\lambda_{j_1})X_{T_r},\cdots, \\ \nonumber
&\frac{1}{\sqrt{2\pi nf_Y(\lambda_{j_\nu})}}\sum_{r=1}^n\cos(r\lambda_{j_\nu})X_{T_r},\frac{1}{\sqrt{2\pi nf_Y(\lambda_{j_\nu })}}\sum_{r=1}^n \sin(r\lambda_{j_\nu })X_{T_r}\Bigg) \\ \nonumber
& \xrightarrow[]{\mathcal{D}}\left(\sqrt{L_{j_1}(d)}(Z_1(j_1),Z_2(j_1)),\cdots,\sqrt{L_{j_\nu}(d)}(Z_1(j_\nu ),Z_2(j_\nu ))\right).
\end{align}

Conditionally on $T_1,\ldots,T_n$, the vector $(X_{T_1},\ldots,X_{T_n})$ is Gaussian, and hence so is 
$Z_n$. Its covariance matrix $\Sigma_T= \var(Z_n| T_1,\ldots,T_n)$ has $(i,k)$ entry 
of the form 
$$
\frac{1}{2\pi n\sqrt{f_Y(\lambda_{j_i})f_Y(\lambda_{j_k})}}\sum_{r=1}^n\sum_{s=1}^n\sigma_X(T_r-T_s)h_{i,k}(r,s)
$$
where
$$
h_{i,k}(r,s)=\cos(r\lambda_{j_i})\cos(s\lambda_{j_k}),\textrm{ or }\cos(r\lambda_{j_i})\sin(s\lambda_{j_k}),\textrm{ or }\sin(r\lambda_{j_i})\sin(s\lambda_{j_k}). 
$$
We prove (\ref{clt}) using the characteristic function: since $X$ and $T$ are independent, for $u\in\mathbb{R}^{2\nu}$, and with $u'$ being the transpose of $u$,
\begin{align*}
 \mathbb{E}( e^{iu'Z_n}) & = \mathbb{E} \left( \mathbb{E}\left( e^{iu'Z_n} \Big | T_1,\ldots,T_n \right) \right) = \mathbb{E} \left( \exp\left(-\frac{1}{2}u'\Sigma_T u\right) \right). 
\end{align*}
As the characteristic function is bounded, it will suffice to show that
\begin{equation}\label{proba}
\Sigma_T\overset{P}{\to}\Sigma,
\end{equation} where $\Sigma$ is the variance-covariance matrix of
$(\sqrt{L_1(j_1)}(Z_1(j_1),Z_2(j_1)),\cdots,\sqrt{L_\nu(j_\nu)}(Z_1(j_\nu ),Z_2(j_\nu )))$.

When $i$ and $k$ are fixed, the form of $h_{i,k}(r,s)$ is the same for all $r$ and $s$ 
and hence $\mathbb{E}\left(\Sigma_T\right)$ will have entries of the form
$$
\frac{1}{2\pi n\sqrt{f_Y(\lambda_{j_i})f_Y(\lambda_{j_k})}}\sum_{r=1}^n\sum_{s=1}^n\mathbb{E}(\sigma_X(T_r-T_s))h_{i,k}(r,s)=
\frac{1}{2\pi n\sqrt{f_Y(\lambda_{j_i})f_Y(\lambda_{j_k})}}\sum_{r=1}^n\sum_{s=1}^n\sigma_Y(r-s)h_{i,k}(r,s)
$$
by \eqref{gammak2}, and therefore $\mathbb{E}(\Sigma_T)\to\Sigma$ by virtue of Theorem 5 of \cite{Hurvich} (the only condition required is second order stationarity of the process $Y_i$ and the behaviour (\ref{spectral-form}) of its spectral density). To complete the proof of (\ref{proba}), it will then suffice to show that
\begin{equation}\label{var}
\textrm{Var}\left(\Sigma_T\right)\to0,
\end{equation}
i.e. the variances of the entries of $\Sigma_T$ converge to zero. By Cauchy-Schwarz inequality, it will be enough to focus on the diagonals. We will treat those diagonals with cosine, as those with sine treat the same way. For some constant $C$ (that may change from one expression to another), we obtain
\begin{align}\label{varcond}
&\nonumber\var\left(
\frac{1}{2\pi nf_Y(\lambda_j)}\sum_{r=1}^n\sum_{s=1}^n\sigma_X(T_r-T_s)\cos(r\lambda_j)\cos(s\lambda_j)\right)\\
&\nonumber\sim\frac{C}{n^{2+4d}}\sum_{r,s,r',s'=1}^n\cov
\left(\sigma_X(T_r-T_s)\cos(r\lambda_j)\cos(s\lambda_j),\sigma_X(T_{r'}-T_{s'})\cos(r'\lambda_j)\cos(s'\lambda_j)\right)\\
&\le\frac{C}{n^{2+4d}}\left(\sum_{r=1}^n\sum_{s=1}^n\sqrt{\var\left(\sigma_X(T_r-T_s)\right)}\right)^2\le\frac{C}{n^{4d}}\left(\sum_{h=1}^n
\sqrt{\var\left(\sigma_X(T_h)\right)}\right)^2\le C\frac{n^{2d}}{n^{4d}}\to0,
\end{align}
using Corollary \ref{var-bound}.
\end{proof}

\section{Inference for the long-memory parameter}
\label{sec:inf}
We still assume in this section that $X= (X_t)_{t\in\mathbb{R}^+}$ is a stationary long memory zero-mean Gaussian process having a spectral density satisfying $H_f$ condition.
Periodogram-based approaches to estimate the long memory parameter $d$
are very popular. Often one requires that the underlying process is
linear or at least is built on martingale difference innovations. The
reader is referred to \cite{beran,Gir} for reviews of some
recent works on this issue, as well as the book edited by \cite{Doukhan}.
The next lemma and its proof show that although it is not a linear process with i.i.d. innovations, the sampled process still satisfies important long memory 4th cumulant conditions. These 4th cumulant conditions will allow us to show both the convergence of an estimator of the memory parameter $d$ and the estimation of the asymptotic variance, necessary for example in the inference about the mean of the original continuous time process $X$.

\begin{Lem}\label{lemme:maj_cumulant}
 Assume that $X$ is a zero mean stationary Gaussian process satisfying $H_f$ and let $Y= (X_{T_n})_{n\in\mathbb{N}}$, where $(T_n)_{n\in\mathbb{N}}$ is a Poisson process.
Then for all $d \in (0,1/2)$, we have 
\begin{equation}\label{double1}
\sup_{h\in \N}\sum_{r,s=0}^n \vert \cum(Y_0,Y_{h},Y_{r},Y_{s})\vert = O( n^{2d}).
\end{equation}
and
\begin{equation}\label{triple1}
\sum_{h,r,s=0}^n \vert \cum(Y_0,Y_{h},Y_{r},Y_{s})\vert = O(n^{4d}\log(n)).
\end{equation}
\end{Lem}

\begin{proof}
The proof is postponed in Appendix. Note that the term $\log(n)$ in the right hand side of \eqref{triple1} is needed only in the particular case $d=1/4$, known to be borderline between weak long memory and strong long memory, as will be seen in the proof. 
\end{proof}

\subsection{Consistency of Local Whittle estimator}
We consider local Whittle estimator of the memory parameter $d$ defined by
\[\widehat{d}_n=\underset{\beta\in[-1/2,1/2]}{\operatorname{argmin}}U_n(\beta) \]
where the contrast function $U_n$ is defined by
\[U_n(\beta)=\log\left(\frac{1}{m_n}\sum_{j=1}^{m_n}\lambda_j^{2\beta}I_n(\lambda_j)\right)-\frac{2\beta}{m_n}\sum_{j=1}^{m_n}\log\lambda_j,\]
 and the bandwidth parameter $m=m_n$ satisfies $m_n\to \infty$ and $m_n=o(n)$.

\begin{thm}\label{dhat}
Suppose ${X}$ is a stationary Gaussian process satisfying condition $H_f$ and that $Y_n=X_{T_n}$, where $T_n$ is a Poisson process with rate 1.
Then,
\begin{equation}\label{dalla1}
\widehat{d}_n\xrightarrow[n\to \infty]{P}d.
\end{equation}
In addition, for $m_n=n^a$, $0<a<1$, we have
\begin{equation}\label{dalla2}
\widehat{d}_n-d=o_P\left(\frac{1}{\log n}\right).
\end{equation}
\end{thm}

\begin{proof}
\noindent 
According to our result (\ref{dala}) and \cite{dalla} (Corollary 1), we have
$$
\hat d-d=O_P\left(m^{-1/2}\log m+\left(\frac{m}{n}\right)^{2d}+r_n\right),
$$
for some remainder $r_n$, which we will be controlling as in what follows, depending on the convergence rate of $m/n$ to zero. To prove (\ref{dalla1}) it will suffice to show that $r_n\to0$.
\paragraph{Case 1 : $\sqrt{n}(\log n)^{4/(1-2d)}=O(m)$. }~
From part (iv) of Corollary 1 of \cite{dalla}, the remainder $r_n$ can be written as 
$$
r_n=\left(\frac{D_n^{**}}{n}\right)^{1/2}\left(\frac{n}{m}\right)^{1-2d}\log^3 n\to0,
$$
where
$$
D_n^{**}=\sup_{h,r\in\N}\sum_{s=1}^n\vert \cum(Y_0,Y_{h},Y_{r},Y_{s})\vert.
$$
We have 
$$
D_n^{**}
\le
\sup_{h\in \N}\sum_{r,s=0}^n \vert \cum(Y_0,Y_{h},Y_{r},Y_{s})\vert,
$$
so that by (\ref{double1}) we get $D_n^{**}=O(n^{2d})$ and hence $r_n=O(1/\log n)$.
\paragraph{Case 2 : $m=O(\sqrt{n}(\log n)^{4/(1-2d)})$. }~ 
We use (iii) of Corollary 1 of \cite{dalla}, 
$$
r_n=\left(\frac{D_n^*}{n}\right)^{1/2}\left(\frac{m}{n}\right)^{2d}\log^2 n\to0,
$$
where
$$
D_n^*=\sum_{h,r,s=0}^n \vert \cum(Y_0,Y_{h},Y_{r},Y_{s})\vert.
$$
According to (\ref{triple1}), $D_n^*=O(n^{4d}\log n)$, and therefore 
$$
r_n=O\left(n^{d-1/2}\left(\log n\right)^{2(1+2d)/(1-2d)}\right)\to0.
$$
This concludes proof of (\ref{dalla1}). \\
To prove (\ref{dalla2}), we show that $r_n=o(1/\log n)$. This is immediate in case (2) above. 
Since $m=n^a$, $0<a<1$, we will be in case (1) if $a>1/2$ and then
$$
r_n=O\left(n^{1/2-d-a(1-2d)}\right)/\log n=o(1/\log n).
$$
\end{proof}

\subsection{Long run variance}

The 4th cumulant condition \eqref{double1} is needed to estimate the long run variance of the sampled process. Such estimation plays a crucial role in many aspects of statistical inference. For example, when it comes to estimating the mean $\mu$ of the original process $X$, as we have from \cite{Phi2}
$$
\left(\textrm{Var}\left(n^{1/2-d}\bar Y_n\right)\right)^{-1/2}n^{1/2-d}(\bar Y_n-\mu)\overset{\mathcal{D}}{\to}\mathcal{N}(0,1)
$$
and hence, it is important to obtain a consistent estimator of the variance above.
Also such estimator is important in testing for short memory versus
long memory or for stationarity versus unit root as such tests involve
V/S type statistics and require estimating the long run
variance (see \cite{giraitis3} and references therein for details). Let us write the spectral density of
 $Y_i$ under the form $f(\lambda)\sim c\vert\lambda\vert^{-2d}$ as
 $\lambda\to 0$. Let
$$
\hat\sigma(h)=\frac{1}{n}\sum_{j=1}^{n-h}(Y_j-\bar Y)(Y_{j+h}-\bar Y)
$$
be the sample covariance function of $Y_i$. Let the asymptotic variance of the normalized sum be
$$
S^2(d) =\underset{n\to\infty}{\lim}\left(\textrm{Var}(n^{1/2-d}\bar Y)\right)=4c\int_{-\infty}^\infty\left(\frac{\sin(\lambda/2)}{\lambda}\right)^2\vert\lambda\vert^{-2d}d\lambda.
$$
Let
$$
\widehat{S}^2(d)=q^{-2d}\left(\hat\sigma(0)+2\sum_{h=1}^q\left(1-\frac{h}{q}\right)\hat\sigma(h)\right).
$$
\begin{prop}
Let $\hat d$ be a consistent estimator for memory parameter $d$ such that $\log(n)(\hat d-d)=o_P(1)$. Let $q\to\infty$ as $n\to\infty$ such that $q=O(\sqrt{n})$.
Then we have
$$
\widehat S^2(\hat d)\overset{P}{\to}S^2(d).
$$
\end{prop}
\begin{proof}
Referring to Theorem 2.2. of \cite{abadir} we just need to verify the cumulant condition
$$
\underset{h}{\sup}\sum_{r,s=1}^n\left\vert\textrm{Cum}(Y_0,Y_h,Y_r,Y_s)\right\vert\le \tilde{c}n^{2d},
$$ for some positive constant $\tilde{c}$.  This is the case according to 
Lemma \ref{lemme:maj_cumulant}.
\end{proof}
\begin{remm}
 A readily available candidate for $\hat d$ above is the Whittle
 estimator for which the $\log(n)$ consistency was established in
 Theorem \ref{dhat}.
\end{remm}

\section*{Appendix : Proof of Lemma \ref{lemme:maj_cumulant} }
\begin{proof}

The proof is essentially based on Corollary \ref{var-bound} and a well known cumulant formula. \\ 
Without loss of generality, we can assume that the Poisson rate is 1.
The process $Y$ is 4th order stationary as the conditional joint distribution of $(Y_k,Y_{k+h},Y_{k+r},Y_{k+s})$ given $(T_1,\dots,T_{k+\max(h,r,s)})$ is a multivariate normal with variance-covariance matrix $M(T_k,T_{k+h},T_{k+r},T_{k+s})$ given by
\begin{equation}
\label{cov-stat}
M(T_k,T_{k+h},T_{k+r},T_{k+s}):=\left( \begin{smallmatrix}
\sigma_X(0)&\sigma_X(T_{k+h}-T_k)&\sigma_X(T_{k+r}-T_k)&\sigma_X(T_{k+s}-T_k)\\
\sigma_X(T_{k+h}-T_k)&\sigma_X(0)&\sigma_X(T_{k+r}-T_{k+h})&\sigma_X(T_{k+s}-T_{k+h})\\
\sigma_X(T_{k+r}-T_k)&\sigma_X(T_{k+r}-T_{k+h})&\sigma_X(0)&\sigma_X(T_{k+s}-T_{k+r})\\
\sigma_X(T_{k+s}-T_k)&\sigma_X(T_{k+s}-T_{k+h})&\sigma_X(T_{k+s}-T_{k+r})&\sigma_X(0)
		\end{smallmatrix}\right) 
		\end{equation}
which is $k$ free. Hence it is enough to establish the lemma when $k=0$.
We apply the total law of cumulance formula, (\cite{Brillinger}), which for the sake of clarity, we remind here: for all random vectors $Z=(Z_1,\ldots,Z_n)'$ and $W$, we have
\begin{equation}\label{brillinger}
\cum(Z)=\sum_{\pi}\cum\left[\cum(X_{\pi_1}\vert W),\ldots,\cum(X_{\pi_b}\vert W)\right]
\end{equation}
where $X_{\pi_j}=(X_i,i\in\pi_j)$, and $\pi_1,\ldots,\pi_b$, ($b=1,\ldots,n$) are the blocks of the permutation $\pi$, and the sum is over all permutations $\pi$ of the set $\{1,2,\ldots,n\}$.\\
But condition on $T$, the process $Y_t$ is jointly zero-mean Gaussian and therefore $\mathbb{E}(Y_t\vert T)=0$ as well as $\cum(Y_i,Y_j,Y_k,Y_\ell\vert T)=\cum(Y_i,Y_j,Y_k\vert T)=0$ for all $i,j,k,\ell$. Hence applying (\ref{brillinger}) to $Y_t$ with $W=T$, only the two-by-two partitions of $\{0,h,r,s\}$ will survive. and since $\cum(U,V)=\textrm{Cov}(U,V)$, we get from (\ref{cov-stat})

\begin{align}\label{cum1}
\cum(Y_0,Y_{h},Y_{r},Y_{s})=&\cov(\sigma_X(T_h),\sigma_X(T_r-T_s))+\cov(\sigma_X(T_r),\sigma_X(T_h-T_s))\nonumber\\
&+\cov(\sigma_X(T_s),\sigma_X(T_r-T_h)).
\end{align}
Note that for $ h< \min(r,s)$, $\cov(\sigma_X(T_h),\sigma_X(T_r-T_s))=0$. 
Moreover 
\begin{align*}
\sum_{1\leq r \leq h \leq s \leq n} |\cov(\sigma_X(T_h),\sigma_X(T_r-T_s)) | 
& \leq \sum_{1\leq r \leq h \leq s \leq n} \var(\sigma_X(T_h))^{1/2} \var( \sigma_X(T_s-T_r))^{1/2}\\ 
& \leq \sum_{1\leq r \leq h \leq s \leq n} h^{-\alpha/2} (1 +s-r)^{-\alpha/2} \\
& \leq h^{-\alpha/2} \sum_{1\leq r \leq h } \sum_{t=1}^n t ^{-\alpha/2} \\
& \leq h^{1-\alpha/2} \begin{cases} n^{1-\alpha/2} = n^{2d -1/2 }& \text{if $d<1/4$} \\ 
 \log(n) & \text{if $d\geq 1/4$} \\ 
\end{cases}\\
& \leq \begin{cases} n^{4d -1 }& \text{if $d<1/4$} \\ 
 \log(n) & \text{if $d\geq 1/4$} \\ 
\end{cases}\\
& \leq Cn^{2d} \; \text{ for all } \; 0< d <1/2. 
\end{align*}
The last configuration is 
\begin{align*}
\sum_{r,s=1}^h | \cov(\sigma_X(T_h),\sigma_X(T_r-T_s))| 
& = \sum_{r,s=1}^h | \var(\sigma_X(T_h))^{1/2} \var( \sigma_X(T_s-T_r))^{1/2}\\ 
& \leq h^{-\alpha/2} \sum_{t=1}^h (h-t) t^{-\alpha/2} \\ 
& \leq h^{-\alpha/2 } \begin{cases} C h^{1-\alpha/2} = Ch^{2d -1/2 }& \text{if $d<1/4$} \\ 
 \log(h) & \text{if $d\geq 1/4$} \\ 
\end{cases} \\ 
& \leq Cn^{2d} \; \text{ for all } \; 0< d <1/2. 
\end{align*}
Therefore uniformly in $h$ we have 
\begin{align*}
\sum_{r,s=1}^n | \cov(\sigma_X(T_h),\sigma_X(T_r-T_s))| \leq Cn^{2d}.
\end{align*}

For the remaining two terms in the right hand side of \eqref{cum1} we have, for fixed $h$, 
\begin{align*}
\sum_{r,s=1}^n | \cov(\sigma_X(T_r),\sigma_X(T_h-T_s))| 
& = \sum_{r,s=1}^n | \cov(\sigma_X(T_s),\sigma_X(T_h-T_r))| \\ 
& \leq \sum_{r,s=1}^n \var(\sigma_X(T_s))^{1/2} \var( \sigma_X(T_h-T_r))^{1/2}\\ 
& \leq \sum_{r,s=1}^n s^{-\alpha/2} (1+|h-r|)^{-\alpha/2} \\ 
& \leq C \begin{cases} n^{2-\alpha} = n^{4d -1 }& \text{if $d<1/4$} \\ 
 \log(n)^2 & \text{if $d\geq 1/4$} \\ 
\end{cases} \\ & \leq Cn^{2d} \; \text{ for all } \; 0< d <1/2. 
\end{align*}

This concludes the proof of (\ref{double1}). 

Let us now prove \eqref{triple1}. Note that 
\begin{align}
\sum_{h,r,s=0}^n \cum(Y_0,Y_{h},Y_{r},Y_{s}) 
& = 3 \sum_{h,r,s=1}^n \cov(\sigma_X(T_h),\sigma_X(T_r-T_s)) \nonumber\\ 
& = 6 \sum_{h=1}^n\sum_{r<s=1}^n \cov(\sigma_X(T_h),\sigma_X(T_r-T_s)). \label{logn}
\end{align}
Moreover, we have 
\begin{align*}
\sum_{h,r,s=1}^n |\cov(\sigma_X(T_h),\sigma_X(T_r-T_s)) | &\leq C
\sum_{h,r,s=1}^n h^{-\alpha/2} (1+|r-s|)^{-\alpha/2} \\
&\leq C
\sum_{h=1}^n h^{-\alpha/2} \sum_{t=1}^n (n-t) t^{-\alpha/2} \\
&\leq C \begin{cases} n n^{2-\alpha} = n^{4d }& \text{if $d<1/4$} \\ 
n \log(n)^2 & \text{if $d> 1/4$} \\ 
\end{cases} \\ 
& \leq Cn^{4d}
\end{align*}
In the particular case $d=1/4$ (where we still have $\alpha=2$), a supplementary term $\log(n)$ is needed in the bound. Indeed 
we split the sum in the right hand side of (\ref{logn}) into 3 configurations. when $1 \leq h \leq r< s \leq n$ the covariance $ \cov(\sigma_X(T_h),\sigma_X(T_r-T_s))$ is zero. When the sum is over $1 \leq r< h \leq s \leq n$, 
we get 
\begin{align*}
 \sum_{1 \leq r< h \leq s \leq n} |\cov(\sigma_X(T_h),\sigma_X(T_r-T_s)) | & \leq C
 \sum_{s=1}^n \sum_{h=1}^s h^{-1 } \sum_{r=1}^{h-1} (s-h + h-r)^{-1} \\
 & \sim C
 \sum_{s=1}^n \sum_{h=1}^s h^{-1 } \left( \log(s) - \log(s-h) \right)\\
 & = -C \sum_{s=1}^n \sum_{h=1}^s (h/s)^{-1 } \log(1-h/s) (1/s) \\
 & \sim -C \sum_{s=1}^n\left(\int_{0}^1 \frac{ \log(1-x)}{x} \, dx\right) = C\frac{\pi^2}{6} n. 
\end{align*}
For the last sum over $1 \leq r < s \leq h \leq n$ (where we will need the $\log$ term) we have 
\begin{align*}
 \sum_{1 \leq r <s \leq h \leq n} |\cov(\sigma_X(T_h),\sigma_X(T_r-T_s)) | & \leq C
 \sum_{h=1}^n h^{-1 } \sum_{s=1}^h \sum_{r=1}^{s-1} (s-r)^{-1} \\
 & = \sum_{h=1}^n h^{-1 } \sum_{t=1}^h (h-t) t^{-1} \\ 
 & = \sum_{h=1}^n \sum_{t=1}^h (1-t/h) t^{-1} \\
 & \sim C \sum_{h=1}^n ( \log(h) -1 ) \sim C n \log(n). 
\end{align*}

This completes the proof of (\ref{triple1}) in Lemma \ref{lemme:maj_cumulant}.
\end{proof}

\end{document}